\documentclass[12pt,a4paper]{amsart}
\usepackage{arXiv}
\tikzset{edgee/.style = {->,> = latex'}}
\newcolumntype{P}[1]{>{\centering\arraybackslash}p{#1}}
\newcolumntype{M}[1]{>{\centering\arraybackslash}m{#1}}

\newcommand\R{\mathbb{R}}

\newcommand{\A}{\mathcal{A}}
\newcommand{\ipa}{\mathrm{L}}

\begin{document}

\title{A combinatorial statistic for labeled threshold graphs}
\author{Priyavrat Deshpande}
\address{Department of Mathematics, Chennai Mathematical Institute, India 603103}
\email{pdeshpande@cmi.ac.in}
\author{Krishna Menon}
\address{Department of Mathematics, Chennai Mathematical Institute, India 603103}
\email{krishnamenon@cmi.ac.in}
\author{Anurag Singh}
\address{Department of Mathematics, Chennai Mathematical Institute, India 603103}
\email{anuragsingh@cmi.ac.in}
\thanks{PD and AS are partially supported by a grant from the Infosys Foundation}

\begin{abstract}
Consider the collection of hyperplanes in $\mathbb{R}^n$ whose defining equations are given by $\{x_i + x_j = 0\mid 1\leq i<j\leq n\}$. 
This arrangement is called the threshold arrangement since its regions are in bijection with labeled threshold graphs on $n$ vertices.
Zaslavsky's theorem implies that the number of regions of this arrangement is the sum of coefficients of the characteristic polynomial of the arrangement. 
In the present article we give a combinatorial meaning to these coefficients as the number of labeled threshold graphs with a certain property, thus answering a question posed by Stanley.
\end{abstract}

\keywords{threshold graph, hyperplane arrangement, finite field method}
\subjclass[2020]{52C35, 05C30, 11B73}
\maketitle

\section{Introduction}\label{intro}
A \emph{hyperplane arrangement} $\A$ is a finite collection of affine hyperplanes (i.e., codimension $1$ subspaces and their translates) in $\R^n$.
A \emph{region} of $\A$ is a connected component of $\R^n\setminus \bigcup \A$.
The number of regions of $\A$ is denoted by $r(\mathcal{A})$.
The poset of non-empty intersections of hyperplanes in an arrangement $\mathcal{A}$ ordered by reverse inclusion is called its \emph{intersection poset} denoted by $\ipa(\A)$.
The ambient space of the arrangement (i.e., $\mathbb{R}^n$) is an element of the intersection poset; considered as the intersection of none of the hyperplanes.
The \emph{characteristic polynomial} of $\A$ is defined as 
\[\chi_\A (t) := \sum_{x\in\ipa(\A)} \mu(\hat{0},x)\, t^{\dim(x)}\]
where $\mu$ is the M\"obius function of the intersection poset and $\hat{0}$ corresponds to $\R^n$.
Using the fact that every interval of the intersection poset of an arrangement is a geometric lattice, we have
\begin{equation}
    \chi_\A(t) = \sum_{i=0}^n (-1)^{n-i} c_i t^i \label{char equation}
\end{equation}
where $c_i$ is a non-negative integer for all $0 \leq i \leq n$ \cite[Corollary 3.4]{stanarr07}.
The characteristic polynomial is a fundamental combinatorial and topological invariant of
the arrangement and plays a significant role throughout the theory of hyperplane arrangements.

In this article our focus is on the enumerative aspects of (rational) arrangements in $\R^n$. 
In that direction we have the following seminal result by Zaslavsky.

\begin{theorem}[\cite{zas75}]\label{zas}
Let $\A$ be an arrangement in $\R^n$. Then the number of regions of $\A$ is given by 
\begin{align*}
   r(\A) &= (-1)^n \chi_{\A}(-1) \\
         &= \sum_{i=0}^n c_i. 
\end{align*}
\end{theorem}

When the regions of an arrangement are in bijection with a certain combinatorially defined set, one could ask if there is a corresponding `statistic' on the set whose distribution is given by the $c_i$'s.
For example, the regions of the braid arrangement in $\R^n$ (whose hyperplanes are given by the equations $x_i-x_j = 0$ for $1\leq i<j\leq n$) correspond to the $n!$ permutations of $[n]$.
The characteristic polynomial of this arrangement is 
$t(t-1)\cdots(t-n+1)$ \cite[Corollary 2.2]{stanarr07}. 
Hence, $c_i$'s are the unsigned Stirling numbers of the first kind. 
Consequently, the distribution of the statistic `number of cycles' on the set of permutations is given by the coefficients of the characteristic polynomial.

In this paper, we consider the following arrangement in $\R^n$ given by
\begin{equation*}
    \mathcal{T}_n := \{x_i + x_j = 0 \mid 1\leq i < j\leq n \}.
\end{equation*}The regions of $\mathcal{T}_n$ are known to be in bijection with labeled threshold graphs on $n$ vertices.
Labeled threshold graphs on $n$ vertices are inductively constructed starting from the empty graph.
Vertices labeled $1,\ldots,n$ are added in a specified order.
At each step, the vertex added is either `dominant' or `recessive'.
A dominant vertex is one that is adjacent to all vertices added before it and a recessive vertex is one that is isolated from all vertices added before it.

\begin{example}\label{constex}
\Cref{tgconst1} shows a construction of a labeled threshold graph on $5$ vertices.
\begin{figure}[H]
    \centering
    \begin{tikzpicture}[scale=0.5]
    \node (2) [circle,draw=black,inner sep=2pt] at (0,0) {\small $2$};
    \node (4) at (0.75,-2) {\small \color{white} 4};
    \node at (1.5,0) {$\rightarrow$};
    \end{tikzpicture}
    \begin{tikzpicture}[scale=0.5]
        \node (2) [circle,draw=black,inner sep=2pt] at (0,0) {\small $2$};
        \node (3) [circle,draw=black,inner sep=2pt] at (2,0) {\small $3$};
        \node (4) at (0.75,-2) {\small \color{white} 4};
        \node at (3.5,0) {$\rightarrow$};
    \end{tikzpicture}
    \begin{tikzpicture}[scale=0.5]
        \node (2) [circle,draw=black,inner sep=2pt] at (0,0) {\small $2$};
        \node (3) [circle,draw=black,inner sep=2pt] at (2,0) {\small $3$};
        \node (1) [circle,draw=black,inner sep=2pt] at (1,2) {\small $1$};
        \node (4) at (0.75,-2) {\small \color{white} 4};
        \node at (3.5,0) {$\rightarrow$};
        \draw (3)--(1)--(2);
    \end{tikzpicture}
    \begin{tikzpicture}[scale=0.5]
        \node (2) [circle,draw=black,inner sep=2pt] at (0,0) {\small $2$};
        \node (3) [circle,draw=black,inner sep=2pt] at (2,0) {\small $3$};
        \node (1) [circle,draw=black,inner sep=2pt] at (1,2) {\small $1$};
        \draw (3)--(1)--(2);
        \node at (3.5,0) {$\rightarrow$};
        \node (4) [circle,draw=black,inner sep=2pt] at (1,-2) {\small $4$};
        \draw (1)--(4);
        \draw (2)--(4);
        \draw (3)--(4);
    \end{tikzpicture}
    \begin{tikzpicture}[scale=0.5]
        \node (2) [circle,draw=black,inner sep=2pt] at (0,0) {\small $2$};
        \node (3) [circle,draw=black,inner sep=2pt] at (2,0) {\small $3$};
        \node (1) [circle,draw=black,inner sep=2pt] at (1,2) {\small $1$};
        \draw (3)--(1)--(2);
        \node (4) [circle,draw=black,inner sep=2pt] at (1,-2) {\small $4$};
        \node (5) [circle,draw=black,inner sep=2pt] at (3.5,0) {\small $5$};
        \draw (1)--(4);
        \draw (2)--(4);
        \draw (3)--(4);
    \end{tikzpicture}
    \caption{Construction of a labeled threshold graph on $5$ vertices.}
    \label{tgconst1}
\end{figure}
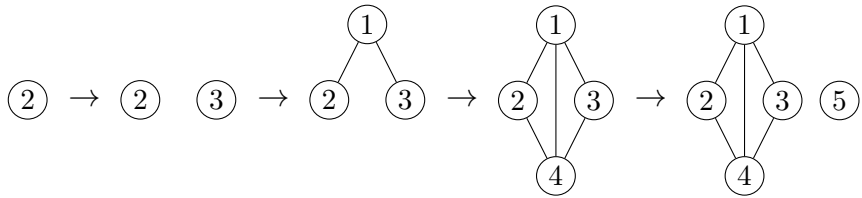
\end{example}

The question of giving combinatorial meaning to the coefficients of the characteristic polynomial of the threshold arrangement as counting certain threshold graphs is an unsolved problem in Stanley's notes \cite[Exercise 5.26]{stanarr07}.
The main aim of this paper is to answer this question.

Any construction of a labeled threshold graph can be expressed as a \textit{signed permutation} (see \cref{signpermdef} for the definition).
The construction associated to a signed permutation is the one where the vertices are added in the order given by the permutation such that any number assigned a $+$ sign is dominant while others are recessive.
For example, the signed permutation associated to the construction in \cref{constex}, if we assume the vertex $2$ is added as a dominant vertex, is $\overset{+}{2}\overset{-}{3}\overset{+}{1}\overset{+}{4}\overset{-}{5}$.
Clearly, there may be more than one way to construct a particular labeled threshold graph.
For example, it can be checked that the construction associated to $\overset{-}{3}\overset{-}{2}\overset{+}{4}\overset{+}{1}\overset{-}{5}$ also yields the same graph as that in \cref{constex}.
Hence, after making a canonical choice of construction for each graph, we can represent labeled threshold graphs by certain signed permutations.
Spiro \cite{spiro20}, using this logic, showed that labeled threshold graphs are in bijection with certain signed permutations called `threshold pairs in standard form' (see \cref{tpsfdef}).

\begin{definition}
In any construction of a labeled threshold graph, an `anchor' is a vertex whose label is smaller than that of all vertices added after it.
An `odd' anchor is one that is of different type, dominant or recessive, from the previous anchor.
The first anchor is considered odd if it is dominant.
\end{definition}

Observe that $1$ is always the first anchor (and hence it is an odd anchor if it is dominant) and the last number is always the last anchor.
The construction corresponding to $\overset{+}{2}\overset{-}{3}\overset{+}{1}\overset{+}{4}\overset{-}{5}$, shown in \cref{tgconst1}, has anchors $1$, $4$ and $5$, but only $1$ and $5$ are odd anchors.

We show that for a given labeled threshold graph, modulo a certain convention about the vertex $1$ which is explained in \cref{tgraphs}, the number of odd anchors does not depend on how it is constructed (\cref{noa}).
The main theorem of this article is:

\begin{theorem*}
Let the number of threshold graphs on $n$ vertices with $j$ odd anchors be $T(n,j)$ and the characteristic polynomial of $\mathcal{T}_n$ be $\chi_{\mathcal{T}_n}(t) = \sum\limits_{j=0}^n (-1)^{n-j} c_j t^j$.
Then,
\begin{equation*}
    c_j = T(n,j) = \sum_{k=j}^n (-1)^{n-k} (S(n,k) + n \cdot S(n-1,k))\cdot a(k,j)
\end{equation*}where $a(k,j)$ is the coefficient of $t^j$ in $(t+1)(t+3)\cdots (t+(2k-1))$ and
$S(n, k)$ are the Stirling numbers of the second kind.
\end{theorem*}
This is proved in \cref{threshcoeff} and \cref{thm} below.
Hence, `number of odd anchors' is a statistic on labeled threshold graphs whose distribution is given by the coefficients of the characteristic polynomial of the threshold arrangement.

The outline of the paper is as follows:
Before studying the threshold arrangement, we give combinatorial meaning to the coefficients of the characteristic polynomial of the type $B$ arrangement in \cref{typb}.
We show that the regions of the type $B$ arrangement are in bijection with signed permutations and that the coefficients count the signed permutations based on the number of `odd cycles'.
Since labeled threshold graphs can be represented as signed permutations, this statistic is later used to define one on the graphs.
In \Cref{prelim}, we derive expressions for the coefficients of the characteristic polynomial of the threshold arrangement using the finite field method.
We also exhibit a bijection between the regions of the threshold arrangement and labeled threshold graphs.
In \Cref{tperms}, we define certain signed permutations in bijection with labeled threshold graphs called `threshold permutations'.
We then show that the coefficients of the characteristic polynomial count the threshold permutations based on the number of odd cycles.
Finally, in \cref{tgraphs}, we translate this interpretation of the coefficients to labeled threshold graphs.

\section{The arrangement of type B}\label{typb}
Before giving combinatorial meaning to the coefficients of the characteristic polynomial of the threshold arrangement, as a warm-up exercise, we do so for the type $B$ arrangement.
The type $B$ arrangement in $\mathbb{R}^n$, which we denote by $\mathcal{B}_n$, has hyperplanes
\begin{equation*}
    \{x_i\pm x_j=0 \mid 1 \leq i < j \leq n\} \cup \{x_i = 0 \mid i \in [n]\}.
\end{equation*}It can be shown that the regions of this arrangement are in bijection with signed permutations of $[n]$.

\begin{remark}\label{signpermdef}
Recall that a signed permutation is a permutation with an assignment of signs, either `$+$' or `$-$', to each of its elements. 
In this article we denote a signed permutation either as a pair $(\pi, w)$ where $\pi = \pi_1\pi_2\cdots\pi_n$ is the permutation on $[n]$ and $w=w_1w_2\ldots w_n$ is the sign sequence, i.e., the sign $w_i$ is assigned to $\pi_i$, or as $\overset{\pm}{\pi_1}\cdots\overset{\pm}{\pi_n}$ where the sign assigned to a number is written above it.
\end{remark}

The type $B$ region associated to a signed permutation $(\pi,w)$ is the one where
\begin{equation*}
    0 < w_1 x_{\pi_1} < w_{2} x_{\pi_{2}} < \cdots < w_{n} x_{\pi_{n}}.
\end{equation*}
Also, using the finite field method, it can be shown \cite[Section 5.1]{stanarr07} that the characteristic polynomial of $\mathcal{B}_n$ is
\begin{equation*}
    \chi_{\mathcal{B}_n}(t)=(t-1)(t-3) \cdots (t-(2n-1)).
\end{equation*}
We know that the absolute values of the coefficients sum to the number of regions, here being $2^n n!$, the number of signed permutation on $[n]$.
We will now exhibit a statistic on signed permutations that induces this break up.

A signed permutation can be split into cycles, just as usual permutations.
We will break a permutation into `signed cycles' in a slightly different manner as follows:
First break the permutation into compartments by drawing a line before the permutation and then repeatedly drawing a line after the least number after the last line drawn.
This is repeated until a line is drawn at the end of the permutation.
We then convert each compartment into a signed cycle by drawing an arrow from each number to the number after it in the compartment (treat the first number in a compartment as the one `after' the last).
These arrows are labeled with the sign assigned to the number it is pointing to.
This method of breaking up a permutation into cycles is very similar to the one used by Stanley in the fundamental bijection on the set of permutations \cite[Section 1.3]{stanbook}.

\begin{example}\label{ex1}
The signed permutation on $[6]$ given by
\begin{equation*}
    \overset{+} 3 \  \overset{+} 1 \  \overset{-} 6 \  \overset{-} 7 \  \overset{-} 5 \  \overset{+} 2 \  \overset{-} 4
\end{equation*}is split into compartments as
\begin{equation*}
    \textcolor{blue}{|} \  \overset{+} 3 \  \overset{+} 1 \  \textcolor{blue}{|} \  \overset{-} 6 \  \overset{-} 7 \  \overset{-} 5 \  \overset{+} 2 \  \textcolor{blue}{|} \  \overset{-} 4 \  \textcolor{blue}{|}
\end{equation*}and is drawn in the cyclic form as shown in \Cref{fig:my_label}.

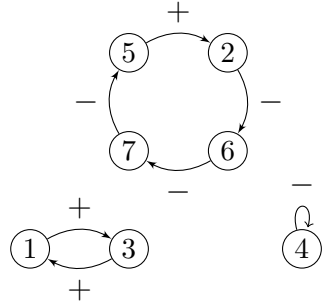
\begin{figure}[H]
    \centering
    \begin{tikzpicture}[scale=0.65]
    \node (1) [circle,draw=black,inner sep=2pt] at (0,0) {\small $1$};
    \node (3) [circle,draw=black,inner sep=2pt] at (2,0) {\small $3$};
    \node (4) [circle,draw=black,inner sep=2pt] at (5.5,0) {\small $4$};
    \node (2) [circle,draw=black,inner sep=2pt] at (4,4) {\small $2$};
    \node (6) [circle,draw=black,inner sep=2pt] at (4,2) {\small $6$};
    \node (7) [circle,draw=black,inner sep=2pt] at (2,2) {\small $7$};
    \node (5) [circle,draw=black,inner sep=2pt] at (2,4) {\small $5$};
    \draw [edgee] (1) to [bend left] node[above] {$+$} (3);
    \draw [edgee] (3) to [bend left] node[below] {$+$} (1);
    \path (4) edge [loop above] node {$-$} (4);
    \draw [edgee] (2) to [bend left] node[right] {$-$} (6);
    \draw [edgee] (6) to [bend left] node[below] {$-$} (7);
    \draw [edgee] (7) to [bend left] node[left] {$-$} (5);
    \draw [edgee] (5) to [bend left] node[above] {$+$} (2);
    \end{tikzpicture}
    \caption{Cycle structure associated to the signed permutation in \Cref{ex1}.}
    \label{fig:my_label}
\end{figure}
\end{example}

To obtain the permutation back from the cycle structure, first find the cycle containing $1$ and start writing the numbers from the number $1$ is pointing to (with sign given on the arrow) in the order specified by the cycle.
Hence we will end with writing $1$.
Next, we delete this cycle and do the same for the cycle containing the least of the remaining numbers.
We repeat this process till all cycles are deleted.

A signed cycle is said to be `odd' if there are an odd number of arrows with a $-$ sign on them.
The term `odd cycle' was defined by Reiner \cite{reiner1993signed}.
The number of odd cycles is just the number of compartments that have an odd number of $-$ signs in them.
For example, the signed permutation in \Cref{ex1} has two odd cycles.
\begin{proposition}
There are
\begin{equation*}
    a(n,j)=\sum_{i=j}^n c(n,i) \cdot \binom{i}{j} \cdot 2^{n-i}
\end{equation*}signed permutations on $[n]$ with $j$ odd cycles, where $c(n,i)$ is the unsigned Stirling number of the first kind; it is the number of permutations of $[n]$ with $i$ cycles.
\end{proposition}
\begin{proof}
To construct a signed permutation on $[n]$ with $j$ odd cycles, we have to choose $j$ cycles from a permutation of $[n]$ and have to assign a $-$ sign to an odd number of arrows in these cycles and assign a $-$ sign to an even number of arrows for the other cycles.
If a cycle has $k$ elements there are $2^{k-1}$ ways to choose an odd (respectively even) number of arrows from it.
Using these facts, we get the above formula.
\end{proof}

\begin{remark}
Alternatively, it can be shown that these numbers satisfy the recurrence relation
\begin{equation*}
    a(n, j) = (2n-1) \cdot a(n-1, j) + a(n-1, j-1)
\end{equation*}with initial conditions $a(n,0)=(2n-1) \cdot (2n-3) \cdots 3 \cdot 1$ and $a(n,n)=1$ for all $n \geq 1$.
This recurrence relation can be used to show that the coefficient of $t^j$ in $\chi_{\mathcal{B}_n}(t)$ gives the number of signed permutations having $j$ odd cycles.
\end{remark}

This sequence of numbers is given in the OEIS \cite{oeis} as \href{http://oeis.org/A028338}{A028338}.

\section{The threshold arrangement}\label{prelim}
We now use the finite field method to compute the characteristic polynomial of $\mathcal{T}_n$ and then extract its coefficients.
We begin by stating a generalization of the finite field method, given by Athanasiadis \cite{athanasiadis1999extended}, in our context. 

\begin{theorem}[{\cite[Theorem 2.1]{athanasiadis1999extended}}]\label{ffm}
If $\mathcal{A}$ is a sub-arrangement of the type $B$ arrangement in $\R^n$, there exists an integer $k$ such that for all odd integers $q$ greater than $k$,
    \begin{equation*}
        \chi_\mathcal{A}(q)=\#(\mathbb{Z}_q^n \setminus V_{\mathcal{A}})
    \end{equation*}
    where $V_{\mathcal{A}}$ is the union of hyperplanes in $\mathbb{Z}_q^n$ obtained by reducing $\mathcal{A}$ mod q.
\end{theorem}

Since the threshold arrangement is a sub-arrangement of the type $B$ arrangement, we can use the above theorem to calculate its characteristic polynomial.

\begin{proposition}
The characteristic polynomial of the threshold arrangement $\mathcal{T}_n$ is
\begin{equation*}
    \chi_{\mathcal{T}_n}(t)=\sum_{k = 1}^n (S(n,k)+n \cdot S(n-1,k))\prod_{i=1}^k(t-(2i-1)).
\end{equation*}
Here $S(n,k)$ are the Stirling numbers of the second kind. 
\end{proposition}
\begin{proof}
Using \cref{ffm}, we see that the characteristic polynomial of $\mathcal{T}_n$ satisfies, for large odd values of $q$,
\begin{equation*}
    \chi_{\mathcal{T}_n}(q)=|\{(a_1,\ldots,a_n) \in \mathbb{Z}_q^n \mid a_i+a_j \neq 0\text{ for all } 1 \leq i < j \leq n\}|.
\end{equation*}This means that we need to count the functions $f:[n] \rightarrow \mathbb{Z}_q$ such that
\begin{enumerate}
    \item $|f^{-1}(0)| \leq 1$.
    \item $f$ can take at most one value from each of the sets
    \begin{equation*}
        \{1,-1\},\{2,-2\},\ldots,\{\frac{q-1}{2},-\frac{q-1}{2}\}.
    \end{equation*}
\end{enumerate}

We split the count into the two cases. If 0 is not attained by $f$, then all values must be from
\begin{equation*}
    \{1,-1\}\cup\{2,-2\}\cup\cdots\cup\{\frac{q-1}{2},-\frac{q-1}{2}\}.
\end{equation*}with at most one value attained in each set. So, there are
\begin{equation*}
    \binom{\frac{q-1}{2}}{k} \cdot 2^k \cdot k! \cdot S(n,k)
\end{equation*}ways for $f$ to attain values from exactly $k$ of these sets. This is because we have $\binom{\frac{q-1}{2}}{k} \cdot 2^k$ ways to choose the $k$ sets and which element of each set $f$ should attain and $k! \cdot S(n,k)$ ways to choose the images of the elements of $[n]$ after making this choice. So the total number of $f$ such that 0 is not attained is
\begin{equation*}
    \sum\limits_{k=1}^{n}\binom{\frac{q-1}{2}}{k} \cdot 2^k \cdot k! \cdot S(n,k).
\end{equation*}

When 0 is attained, there are $n$ ways to choose which element of $[n]$ gets mapped to 0 and using a similar logic for choosing the images of the other elements, we get that the total number of $f$ where 0 is attained is
\begin{equation*}
    n \cdot \sum\limits_{k=1}^{n-1}\binom{\frac{q-1}{2}}{k} \cdot 2^k \cdot k! \cdot S(n-1,k).
\end{equation*}

So we get that for large odd values of $q$,
\begin{align*}
    \chi_{\mathcal{T}_n}(q)&=\sum\limits_{k=1}^{n}\binom{\frac{q-1}{2}}{k} \cdot 2^k \cdot k! \cdot S(n,k) + n \cdot \sum\limits_{k=1}^{n-1}\binom{\frac{q-1}{2}}{k} \cdot 2^k \cdot k! \cdot S(n-1,k)\\
    &=\sum_{k = 1}^n (S(n,k)+n \cdot S(n-1,k))\prod_{i=1}^k(q-(2i-1)).
\end{align*}Since $\chi_{\mathcal{T}_n}$ is a polynomial, we get the required result.
\end{proof}

\begin{corollary}\label{threshcoeff}
The coefficient of $t^j$ in $\chi_{\mathcal{T}_n}(t)$ is
\begin{equation*}
    \sum_{k=j}^n (-1)^{k-j} (S(n,k)+n \cdot S(n-1,k)) a(k,j)
\end{equation*}
where $a(k,j)$ is the coefficient of $t^j$ in $(t+1)(t+3)\cdots (t+(2k-1))$, which, by observations made in \Cref{typb}, is the number of signed permutations on $[k]$ with $j$ odd cycles.
\end{corollary}

\begin{remark}
Using \cref{zas}, it can be shown that the number of regions of $\mathcal{T}_n$ is
\begin{equation*}
    2\cdot (a(n)-n \cdot a(n-1))
\end{equation*}where $a(n)$ is the $n^{th}$ ordered Bell number.
It can be shown that this expression is equal to
\begin{equation*}
    \sum\limits_{k=1}^{n-1} 2^k(n-k)A(n-1,k-1)
\end{equation*}where $A(n,k)$ are the Eulerian numbers.
This is a known expression for the number of labeled threshold graphs on $n$ vertices \cite{spiro20}.
This sequence of numbers is given in the OEIS \cite{oeis} as \href{http://oeis.org/A005840}{A005840}.
\end{remark}

The list of characteristic polynomials $\chi_{\mathcal{T}_n}(t)$ and corresponding number of regions of $\mathcal{T}_n$ for $n \leq 10$ is given in \cref{tab:my_label}.

\begin{table}[H]
    \addtolength{\tabcolsep}{10pt}
    \renewcommand{\arraystretch}{1.25}
    \centering
    \begin{tabularx}{\textwidth}{|c|X|c|}
        \hline
         $n$ & \hspace{4cm}$\chi_{\mathcal{T}_n}(t)$ & {$r(\mathcal{T}_n)$} \\
        \hline
        2 & $t^2-t$ & 2\\
        \hline
        3 & $t^3-3t^2+3t-1$ & 8\\
        \hline
        4 & $t^4-6t^3+15t^2-17t+7$ & 46\\
        \hline
        5 & $t^5-10t^4+45t^3-105t^2+120t-51$ & 332\\
        \hline
        6 & $t^6-15t^5+105t^4-410t^3+900t^2-1012t+431$ & 2874\\
        \hline
        7 & $t^7-21t^6+210t^5-1225t^4+4340t^3-9058t^2+9961t-4208$ & 29024\\
        \hline
        8 & $ t^8-28t^7+378t^6-3066t^5+15855t^4-52234t^3+104433t^2-112163t+46824$ & 334982\\
        \hline
        9 & $ t^9-36t^8+630t^7-6762t^6+47817t^5-226380t^4+703815t^3-1355427t^2 +1422483t-586141$ & 4349492\\
        \hline
        10 & $t^{10}-45t^9+990t^8-13560t^7+125265t^6-801507t^5+3541125t^4-10491450t^3+19546335t^2-20068391t+8161237$ & 62749906\\
        \hline
    \end{tabularx}
    \caption{Characteristic polynomial and the number of regions of $\mathcal{T}_n$ for $n \leq 10$.}
    \label{tab:my_label}
\end{table}

We now recall the bijection between the regions of the threshold arrangement and labeled threshold graphs.

\begin{definition}
A threshold graph is defined recursively as follows:
\begin{enumerate}
    \item The empty graph is a threshold graph.
    \item A graph obtained by adding an isolated vertex to a threshold graph is a threshold graph.
    Such a vertex will be called a recessive vertex of the graph.
    \item A graph obtained by adding a vertex adjacent to all vertices of a threshold graph is a threshold graph.
    Such a vertex will be called a dominant vertex of the graph.
\end{enumerate}
\end{definition}

\begin{definition}
A labeled threshold graph is a threshold graph having $n$ vertices with vertices labeled distinctly using $[n]$.
Such graphs can also be defined using the construction above, where $n$ vertices are used in the construction, all labeled distinctly.
\end{definition}

The first vertex that is added in the construction of a labeled threshold graph can be considered as either dominant or recessive.
If two vertices are added consecutively and are of the same type, the order in which they are added does not matter.

\begin{definition}\label{tpsfdef}
A threshold pair in standard form of size $n$ is a signed permutation on $[n]$ that satisfies the following conditions:
\begin{enumerate}
    \item The signs of the first two numbers are the same.
    \item If the sign of any two consecutive numbers are the same, the first number is smaller.
\end{enumerate}
\end{definition}

\begin{example}
$\overset{-}{2}\ \overset{-}{3}\ \overset{-}{5}\ \overset{+}{1}\ \overset{+}{4}\ \overset{-}{6}$ is a threshold pair in standard form whereas $\overset{+}{2}\ \overset{-}{3}\ \overset{-}{5}\ \overset{+}{4}\ \overset{+}{1}\ \overset{-}{6}$ is not.
\end{example}

Spiro \cite{spiro20} showed that labeled threshold graphs with $n$ vertices are in bijection with threshold pairs in standard form of size $n$.
A threshold pair in standard form $(\pi,w)$ corresponds to the labeled threshold graph obtained by adding vertices in the order given by the permutation such that:
\begin{enumerate}
    \item The vertex $\pi_i$ is added as a dominant vertex if $w_i = +$.
    \item The vertex $\pi_i$ is added as a recessive vertex if $w_i = -$.
\end{enumerate}

We will use these threshold pairs in standard form to show that the regions of the threshold arrangement are in bijection with the labeled threshold graphs.

\begin{lemma}
The regions of the arrangement $\mathcal{T}_n$ are in bijection with the threshold pairs in standard form of size $n$.
\end{lemma}
\begin{proof}
Given a threshold pair in standard form $(\pi,w)$, we associate the region of $\mathcal{T}_n$ where for any $1 \leq i < j \leq n$, we have:
\begin{enumerate}
    \item $x_{\pi_i} + x_{\pi_j} > 0$ if $w_j = +$.
    \item $x_{\pi_i} + x_{\pi_j} < 0$ if $w_j = -$.
\end{enumerate}
This does correspond to a region since a point $(a_1,a_2,\ldots,a_n)$ in $\mathbb{R}^n$ with $a_i$ nonzero, sign of $a_i$ being $w_i$, and $|a_{\pi_{i}}| < |a_{\pi_{i+1}}|$ for all $i \in [n-1]$ satisfies the inequalities of the threshold region.
It can be checked that different threshold pairs in standard form are associated with different threshold regions.

Conversely, given a threshold region $R$, it is possible to find a point $(a_1,a_2,\ldots,a_n)$ in the region with all coordinates nonzero and having distinct absolute values.
Let $b_i$ denote the absolute value of $a_i$ and $w'_i$ be the sign of $a_i$ for all $i \in [n]$.
Let $\pi'_1\pi'_2\cdots\pi'_n$ be the permutation on $[n]$ such that $0 < b_{\pi'_1} < b_{\pi'_2} < \cdots < b_{\pi'_n}$.
Let $\pi'$ be the signed permutation $\pi'_1\pi'_2\cdots\pi'_n$ where $\pi'_i$ is assigned the sign $w'_{\pi_i}$.
In fact, we can assume $w'_1=w'_2$.

If $w'_1 \neq w'_2$, consider the point $(c_1,c_2,\ldots,c_n)$ where $c_i=a_i$ if $i \neq \pi'_1$ and $c_{\pi'_1}=-a_{\pi'_1}$.
We claim that this point is also in $R$ by showing that $a_i+a_j>0$ if and only if $c_i+c_j>0$ for any $1 \leq i < j \leq n$.
This is clear if neither $i$ nor $j$ is $\pi'_1$.
The number $a_{\pi'_1}$ has the least absolute value, hence for any $j \neq \pi'_1$, $a_j$ does not lie between $-a_{\pi'_a}$ and $a_{\pi'_1}$, i.e., between $c_{\pi'_1}$ and $a_{\pi'_1}$.
Consequently for any $j \neq \pi'_1$, $a_j>-a_{\pi'_1}$ if and only if $c_j>-c_{\pi'_1}$.
So $(c_1,c_2,\ldots,c_n)$ lies in $R$ and the signed permutation constructed as described above has the first two signs equal.

The threshold pair in standard form obtained by reordering all maximal strings of numbers of the same sign in $(\pi',w')$ to ascending order is the one associated with this region.
Let $(\pi,w)$ be this threshold pair in standard from associated to $R$.

Let $1 \leq i < j \leq n$.
If $w_j=+$, this means that $a_{\pi_j}>0$ and hence $b_{\pi_j}=a_{\pi_j}$.
If $w_i=+$, i.e., $a_{\pi_i}>0$, we automatically get $a_{\pi_j} > -a_{\pi_i}$.
If $w_i=-$, i.e., $a_{\pi_i}<0$, we must have $b_{\pi_j}>b_{\pi_i}$ since $\pi_i$ and $\pi_j$ have different signs and hence have the same relative position in $\pi$ as they did in $\pi'$.
In any case, we have $a_{\pi_i}+a_{\pi_j}>0$.
This means that $R$ should lie in the half-space $x_{\pi_i}+x_{\pi_j}>0$.
Similarly, if $w_j=-$, we can show that $R$ lies in the half-space $x_{\pi_i}+x_{\pi_j}<0$.
This shows that $R$ is the same region as the one associated to $(\pi,w)$ as described in the first paragraph.
Hence the two maps described above are inverses.
\end{proof}

\begin{remark}
Combining the bijections described above, we obtain a bijection between regions of $\mathcal{T}_n$ and labeled threshold graphs on $n$ vertices where, for any $1 \leq i < j \leq n$, there is an edge between the vertices $i$ and $j$ in the labeled threshold graph if and only if $x_i+x_j>0$ in the corresponding region.
\end{remark}

\section{Threshold permutations}\label{tperms}
As mentioned in \cref{intro}, we will be using different signed permutations, which we call `threshold permutations', to represent labeled threshold graphs.
Though it would be more accurate to called them `threshold signed permutations', we drop the term `signed' for convenience.
These are defined as follows:

\begin{definition}\label{tpermdef}
A threshold permutation of size $n$ is a signed permutation on $[n]$ that satisfies the following conditions:
\begin{enumerate}
    \item The first number is less than the second.
    \item If the first number is $1$, it is assigned a $-$ sign.
    \item If the first number is not $1$, the second number is assigned a $-$ sign.
    \item Any number that is assigned a $+$ sign is larger than the number before it.
\end{enumerate}
\end{definition}

We show that threshold permutations correspond to labeled threshold graphs by exhibiting a bijection between them and threshold pairs in standard form.
Given a threshold pair in standard form $(\pi,w)$, we associate a threshold permutation with the same underlying permutation as follows:
\begin{enumerate}
    \item For every $\pi_i$ where $i > 2$, we assign the sign $+$ if $w_{i-1}=w_i$ and a $-$ sign otherwise.
    \item If $\pi_1 \neq 1$, the sign of $\pi_1$ remains the same and a $-$ sign is assigned to $\pi_2$.
    \item If $\pi_1 =1$, $\pi_2$ is assigned the sign $w_1$ and $\pi_1=1$ is assigned a $-$ sign.
\end{enumerate}

\begin{example}\label{ex2}
\begin{enumerate}
    \item The threshold pair in standard form given by $\overset{-}{2}\ \overset{-}{3}\ \overset{-}{5}\ \overset{+}{1}\ \overset{+}{4}\ \overset{-}{6}$ has corresponding threshold permutation $\overset{-} 2 \  \overset{-} 3\  \overset{+} 5\  \overset{-} 1\  \overset{+} 4\  \overset{-} 6$.
    \item The threshold pair in standard form given by $\overset{+}{1}\ \overset{+}{3}\ \overset{-}{2}\ \overset{-}{5}\ \overset{+}{4}\ \overset{+}{6}\ \overset{+}{7}$ has corresponding threshold permutation $\overset{-} 1\  \overset{+} 3\  \overset{-} 2\  \overset{+} 5\  \overset{-} 4\  \overset{+} 6\  \overset{+} 7$.
\end{enumerate}
\end{example}

Given a threshold permutation $\pi_1\pi_2\cdots\pi_n$ where the sign assigned to $\pi_i$ is $w_i$, we can obtain the threshold pair in standard form associated with it, having the same underlying permutation, as follows:

Let $j_1 < j_2 < j_3 < \cdots$ be those $i \in [n]$ such that $w_i = -$ and $i > 2$.
\begin{enumerate}
    \item If $\pi_1 \neq 1$, the sign $w_1$ is assigned to $\pi_i$ for $i < j_1$, $j_2 \leq  i < j_3$, $j_4 \leq i < j_5$, $\ldots$
    All other terms of the permutation are assigned the opposite of $w_1$.
    \item If $\pi_1 = 1$, the sign $w_2$ is assigned to $\pi_i$ for $i < j_1$, $j_2 \leq  i < j_3$, $j_4 \leq i < j_5$, $\ldots$
    All other terms of the permutation are assigned the opposite of $w_2$.
\end{enumerate}
It can be checked that the above procedure applied to the threshold permutations in \Cref{ex2} produces the original threshold pairs in standard form, and more generally that it reverses the process defined before \Cref{ex2}.
So, the regions of the threshold arrangement and threshold graphs are both in bijection with threshold permutations.

We can compartmentalize a threshold permutation just as done for usual signed permutations in \Cref{typb} and consider the number of odd cycles it has.

\begin{example}
The threshold permutation
\begin{equation*}
    \overset{-} 2 \  \overset{-} 3\  \overset{+} 5\  \overset{-} 1\  \overset{+} 4\  \overset{-} 6
\end{equation*}is split into compartments as
\begin{equation*}
    \textcolor{blue}{|} \  \overset{-} 2 \  \overset{-} 3\  \overset{+} 5\  \overset{-} 1 \  \textcolor{blue}{|} \  \overset{+} 4\  \textcolor{blue}{|} \  \overset{-} 6 \  \textcolor{blue}{|}.
\end{equation*}Hence this threshold permutation has $2$ odd cycles.
\end{example}

We can now state the main theorem of this section.

\begin{theorem}\label{thm1}
The absolute value of the coefficient of $t^j$ in the polynomial $\chi_{\mathcal{T}_n}(t)$ is the number of threshold permutations of size $n$ with exactly $j$ odd cycles.
\end{theorem}

To prove this theorem, we will count the threshold permutations subject to odd cycle number and show that this matches with the coefficients of the characteristic polynomial.
This will be done as follows:
\begin{enumerate}
    \item In \cref{nperms}, we count a larger class of signed permutations,  which we call `normal permutations', subject to the number of odd cycles.
    \item In \cref{thm1prf}, we show that normal permutations on $[n-1]$ can be used to count those normal permutations on $[n]$ that are not threshold permutations (subject to the number of odd cycles).
    This gives us an expression for the number of threshold permutations with a given number of odd cycles and hence proves \cref{thm1}.
\end{enumerate}

\subsection{Normal permutations}\label{nperms}
\begin{definition}
A normal permutation is a signed permutation such that any number assigned a $+$ sign is larger than the one before it.
We represent the number of normal permutations on $[n]$ with $j$ odd cycles as $N(n,j)$.
\end{definition}

\begin{example}
The signed permutations given by $\overset{-} 3\  \overset{-} 1\  \overset{+} 5\  \overset{+} 7\  \overset{-} 2\  \overset{+} 4\  \overset{+} 6$ is a normal permutation (with $1$ odd cycle) whereas $\overset{-} 3\  \overset{+} 7\  \overset{+} 1\  \overset{-} 5\  \overset{-} 4\  \overset{+} 6\  \overset{-} 2$ is not.
\end{example}

\begin{remark}
Note that all threshold permutations are normal permutations.
\end{remark}

\begin{proposition}\label{numbernormal}
The number of normal permutations on $[n]$ having $j$ odd cycles is
\begin{equation}\label{eq1}
    N(n,j) = \sum_{k=j}^n (-1)^{n-k}S(n,k) \cdot a(k,j)
\end{equation}where $a(k,j)$ is defined in \Cref{typb} as the number of signed permutations of $[k]$ with $j$ odd cycles.
\end{proposition}

To prove \Cref{numbernormal}, we define a sign-reversing involution on pairs of the form $(P,(\pi,w))$ where $P$ is a partition of $[n]$ with $k$ blocks and $(\pi,w)$ is a signed permutation of size $k$ with $j$ odd cycles.
If $P=B_1/B_2/\cdots/B_k$, where $\operatorname{min}(B_1) < \operatorname{min}(B_2) < \cdots < \operatorname{min}(B_k)$ and $(\pi,w)=(\pi_1\pi_2\cdots\pi_k,w_1w_2\cdots w_k)$, we represent the pair $(P,(\pi,w))$ as
\begin{equation*}
    w_1\ B_{\pi_1}\ w_2\ B_{\pi_2}\ \cdots\ w_k\ B_{\pi_k}.
\end{equation*}

\begin{example}
The pair given by
\begin{equation*}
    (\{1,3\}/\{2,6\}/\{4\}/\{5\}/\{7\}, \overset{-}2 \  \overset{-}4 \  \overset{+}1 \  \overset{-}5 \  \overset{-}3)
\end{equation*}has representation
\begin{equation*}
    -\  \{2,6\}\  -\  \{5\}\  +\  \{1,3\}\  -\  \{7\}\  -\  \{4\}.
\end{equation*}
\end{example}

The representation of $(P,(\pi,w))$ is said to have size $n$ if $P$ is a partition of $[n]$, $k$ blocks if $P$ has $k$ blocks, and $j$ odd cycles if $(\pi,w)$ has $j$ odd cycles.
A \textit{part} of a representation is a portion of the representation, after deleting the first sign, that is between consecutive $-$ signs, before the first $-$ sign, or after that last $-$ sign.

\begin{example}
The representation given by
\begin{equation*}
    -\  \{2,6\}\  -\  \{5\}\  +\  \{1,3\}\  -\  \{7\}\  -\  \{4\}
\end{equation*}is of size $7$, has $5$ blocks, and has $0$ odd cycles.
The parts of the representation are
\begin{equation*}
    \{2,6\},\ \{5\}\  +\  \{1,3\},\ \{7\}\ \text{and}\ \{4\}.
\end{equation*}
\end{example}

\begin{definition}
A normal representation is one in which each part is of the form
\begin{equation*}
    \{a_1\}\ +\ \{a_2\}\ +\ \cdots\ +\ \{a_m\}
\end{equation*}where $a_1<a_2<\cdots<a_m$.
\end{definition}

It is clear that normal representations correspond to those pairs $(P,(\pi,w))$ where $P=\{1\}/\{2\}/\cdots/\{n\}$ and $(\pi,w)$ is a normal permutation.
Hence, $N(n,j)$ is the number of normal representations of size $n$ with $j$ odd cycles.

\begin{proof}[Proof of \Cref{numbernormal}]
We have to show that the summation given in \eqref{eq1} counts the number of normal representations of size $n$ with $j$ odd cycles.
To do so, we rewrite \eqref{eq1} as
\begin{equation*}
    N(n,j)\ + \sum_{k \not\equiv n (mod\ 2)} S(n,k) \cdot a(k,j)\ = \sum_{k \equiv n (mod\ 2)} S(n,k) \cdot a(k,j).
\end{equation*}
The term on the right hand side is the number of representations of size $n$ with $j$ odd cycles whose number of blocks has the same parity as $n$.
Similarly, the term on the left hand side after $N(n,j)$ is the number of representations of size $n$ with $j$ odd cycles whose number of blocks has the opposite parity as $n$.
Since normal representations have $n$ blocks, they are representations of the type counted by the term on the right hand side.
Hence, we can prove our result by obtaining a sign reversing involution (with respect to number of blocks) on the set of representations of size $n$ with $j$ odd cycles, that are not normal.

We define such an involution as follows:
For any representation that is not normal, find the first part that is not of the form
\begin{equation*}
    \{a_1\}\   +\   \{a_2\}\   +\   \cdots\   +\  \{a_m\}
\end{equation*}where $a_1 < a_2 < \cdots < a_m$.
Suppose the numbers in the blocks of this part are $a_1 < a_2 < \cdots < a_m$ and $a_i$ is the first number such that one of the following holds:
\begin{enumerate}
    \item $a_i$ is in a block $B$ of size greater than 1, or
    \item $\{a_i\}$ is a block but there is some element in the block $B'$ before it that is larger than $a_i$.
\end{enumerate}
In the first case, we split $B$ in the representation as $B \setminus \{a_i\} + \{a_i\}$.
In the second case, we put $a_i$ into $B'$, i.e., replace $B' + \{a_i\}$ by $B'\cup \{a_i\}$.

\begin{example}
The representation of size $7$ with $2$ odd cycles given by
\begin{equation*}
    +\  \{1\} \  + \  \{3\} \  - \  \{2\}\  +\  \{5\}\  +\  \{4\}\  -\  \{6,7\}
\end{equation*}is matched with the representation of size $7$ with $2$ odd cycles given by
\begin{equation*}
    +\  \{1\} \  + \  \{3\} \  - \  \{2\}\  +\  \{4,5\}\  -\  \{6,7\}.
\end{equation*}
\end{example}

It can be shown that this involution satisfies our requirements.
Hence the summation given in \eqref{eq1} is the number of normal permutations of size $n$ with $j$ odd cycles.
\end{proof}

\subsection{Proof of Theorem \ref{thm1}}\label{thm1prf}

We now show that normal permutations on $[n]$ having $j$ odd cycles that are not threshold permutations are in bijection with pairs of the form $(b,\pi)$ where $b$ is a number in $[n]$ and $\pi$ is a normal permutation on $[n-1]$ with $j$ odd cycles.
This will then give us an expression for the number of threshold permutations on $[n]$ having $j$ odd cycles as $N(n,j) - n \cdot N(n-1,j)$.
Using the expression for $N(n,j)$ given by \eqref{eq1}, we get
\begin{equation*}
    N(n,j) - n \cdot N(n-1,j) = \sum_{k=j}^n (-1)^{n-k} (S(n,k) + n \cdot S(n-1,k))\cdot a(k,j).
\end{equation*}From this and the expression for the coefficients of the characteristic polynomial given in \cref{threshcoeff}, \cref{thm1} follows.
So it only remains to prove the following lemma.

\begin{lemma}\label{thm2}
Pairs of the form $(b,\pi)$ where $b$ is a number in $[n]$ and $\pi$ is a normal permutation on $[n-1]$ with $j$ odd cycles are in bijection with normal permutations on $[n]$ with $j$ odd cycles that are not threshold permutations.
\end{lemma}
\begin{proof}
For any such pair $(b,\pi)$, we will call $b$ the special number and consider $\pi$ as a normal permutation on $[n] \setminus \{b\}$ using the order-preserving bijection from $[n-1]$ to $[n] \setminus \{b\}$.

We now match a given such pair with some normal permutation of size $n$ with $j$ odd cycles.
Let the special number of the given pair be $b$ and the first number in the normal permutation on $[n] \setminus \{b\}$ be $a_1$.
We now consider several cases:
\begin{enumerate}
    \item If $b > a_1$, we match the normal permutation of size $n$ which is obtained by inserting $b$ before $a_1$, assigning it the opposite sign as $a_1$, and then assigning $a_1$ a $-$ sign.
    That is, if the normal permutation of size $n-1$ is
    \begin{equation*}
        \overset{\pm}{a_1}\   \cdots\   \text{ and $\  b > a_1$,}
    \end{equation*}then we match this the normal permutation of size $n$ given by
    \begin{equation*}
        \overset{\mp}{b}\   \overset{-}{a_1}\   \cdots.
    \end{equation*}This normal permutation also has $j$ odd cycles since the parity of the number of $-$ signs in the first compartment remains the same and other compartments do not change.
    \item If $b < a_1$ and $b \neq 1$, we match the normal permutation of size $n$ which is obtained by inserting $b$ before $a_1$, assigning it the same sign as $a_1$, and then assigning $a_1$ a $+$ sign.
    That is, if the normal permutation of size $n-1$ is
    \begin{equation*}
        \overset{\pm}{a_1}\   \cdots\   \text{ and $\ 1 < b < a_1$,}
    \end{equation*}then we match this with the normal permutation of size $n$ given by
    \begin{equation*}
        \overset{\pm}{b} \   \overset{+}{a_1}\   \cdots.
    \end{equation*}This normal permutation also has $j$ odd cycles since $b \neq 1$ and the number of $-$ signs in each compartment remains unchanged.
    \item If $b=1$, we insert $\overset{+}{1}$ at the beginning of the normal permutation of size $n-1$.
    That is, if the normal permutation of size $n-1$ is
    \begin{equation*}
        \overset{\pm}{a_1}\   \cdots\   \text{ and $\  b = 1$,}
    \end{equation*}then we match this with the normal permutation of size $n$ given by
    \begin{equation*}
        \overset{+}{1}\  \overset{\pm}{a_1}\   \cdots.
    \end{equation*}This normal permutation also has $j$ odd cycles since we have just a compartment containing only $\overset{+}{1}$.
\end{enumerate}

\begin{example}
The pair having special number $3$ and normal permutation of size $6$ given by
\begin{equation*}
    \overset{+} 1\  \overset{-} 5\  \overset{+} 7\  \overset{-} 2\  \overset{+} 4\  \overset{-} 6
\end{equation*}
is matched with the normal permutation of size $7$ given by
\begin{equation*}
    \overset{-} 3\  \overset{-} 1\  \overset{-} 5\  \overset{+} 7\  \overset{-} 2\  \overset{+} 4\  \overset{-} 6.
\end{equation*}It can be checked that both have $1$ odd cycle.
\end{example}

The matching described above gives a bijection between pairs of the type mentioned in the statement of the lemma and normal permutations on $[n]$ with $j$ odd cycles of the following types:
\begin{enumerate}
    \item The first number is greater than the second.
    \item The first number is not $1$ and the second is assigned a $+$ sign.
    \item The first number is $1$ which is assigned a $+$ sign.
\end{enumerate}But these are precisely those normal permutations that are not threshold permutations.
\end{proof}

The threshold pairs in standard form of size $2$ and $3$, corresponding threshold permutations, and number of odd cycles are given in \cref{tab:my_label2} and \cref{tab:my_label3} respectively.
Using \Cref{thm1} and the contents of these tables, we get $\chi_{\mathcal{T}_2}(t)=t^2-t$ and $\chi_{\mathcal{T}_3}(t)=t^3-3t^2+3t-1$, which agrees with the expressions in \Cref{tab:my_label}.
\FloatBarrier
\begin{table}[H]
    \addtolength{\tabcolsep}{10pt}
    \renewcommand{\arraystretch}{1.25}
    \centering
    \begin{tabular}{|M{3cm}|M{4cm}|M{2.5cm}|}
        \hline
        Threshold pair in standard form& Corresponding threshold permutation & Number of odd cycles\\
        \hline
        $\overset{-}{1}\ \overset{-}{2}$ & $\overset{-}{1}\ \overset{-}{2}$ & 2\\
        \hline
        $\overset{+}{1}\ \overset{+}{2}$ & $\overset{-}{1}\ \overset{+}{2}$ & 1\\
        \hline
    \end{tabular}
    \caption{Threshold pairs in standard form of size $2$, threshold permutations and number of odd cycles.}
    \label{tab:my_label2}
\end{table}
\FloatBarrier
\begin{table}[H]
    \addtolength{\tabcolsep}{10pt}
    \renewcommand{\arraystretch}{1.25}
    \centering
    \begin{tabular}{|M{3cm}|M{4cm}|M{2.5cm}|}
        \hline
        Threshold pair in standard form& Corresponding threshold permutation & Number of odd cycles\\
        \hline
        $\overset{-}{1}\ \overset{-}{2}\ \overset{+}{3}$ & $\overset{-}{1}\ \overset{-}{2}\ \overset{-}{3}$ & 3\\
        \hline
        $\overset{+}{1}\ \overset{+}{2}\ \overset{-}{3}$ & $\overset{-}{1}\ \overset{+}{2}\ \overset{-}{3}$ & 2\\
        \hline
        $\overset{+}{1}\ \overset{+}{3}\ \overset{-}{2}$ & $\overset{-}{1}\ \overset{+}{3}\ \overset{-}{2}$ & 2\\
        \hline
        $\overset{-}{1}\ \overset{-}{2}\ \overset{-}{3}$ & $\overset{-}{1}\ \overset{-}{2}\ \overset{+}{3}$ & 2\\
        \hline
        $\overset{-}{1}\ \overset{-}{3}\ \overset{+}{2}$ & $\overset{-}{1}\ \overset{-}{3}\ \overset{-}{2}$ & 1\\
        \hline
        $\overset{-}{2}\ \overset{-}{3}\ \overset{+}{1}$ & $\overset{-}{2}\ \overset{-}{3}\ \overset{-}{1}$ & 1\\
        \hline
        $\overset{+}{1}\ \overset{+}{2}\ \overset{+}{3}$ & $\overset{-}{1}\ \overset{+}{2}\ \overset{+}{3}$ & 1\\
        \hline
        $\overset{+}{2}\ \overset{+}{3}\ \overset{-}{1}$ & $\overset{+}{2}\ \overset{-}{3}\ \overset{-}{1}$ & 0\\
        \hline
    \end{tabular}
    \caption{Threshold pairs in standard form of size $3$, threshold permutations and number of odd cycles.}
    \label{tab:my_label3}
\end{table}

\section{Threshold graphs}\label{tgraphs}
We now use the bijection between threshold permutations and labeled threshold graphs to obtain a corresponding statistic on the graphs.
Before doing so, we make a few definitions and conventions regarding the construction of labeled threshold graphs.

\begin{definition}
The set of numbers forming a maximal string having the same sign in a threshold pair in standard form is called a block.
If the sign of the numbers in a block is positive, it is called a dominant block and if it is negative, it is called a recessive block.
\end{definition}
We order the blocks according to the order in which they appear in the permutation.
We can define blocks for any labeled threshold graph using the bijection with threshold pairs in standard form.
By the definition of a threshold pair in standard form, the first block always has size larger than $1$.

\begin{example}
The blocks in the threshold pair in standard form given by $\overset{-}{2}\ \overset{-}{3}\ \overset{-}{5}\ \overset{+}{1}\ \overset{+}{4}\ \overset{-}{6}$ are $\{2,3,5\}$, $\{1,4\}$ and $\{6\}$.
\end{example}

Recall that the labeled threshold graph corresponding to a threshold pair in standard form $(\pi,w)$ is obtained by adding vertices in the order given by the permutation such that:
\begin{enumerate}
    \item The vertex $\pi_i$ is added as a dominant vertex if $w_i = +$.
    \item The vertex $\pi_i$ is added as a recessive vertex if $w_i = -$.
\end{enumerate}

As mentioned in \cref{intro}, any signed permutation corresponds to a construction of a labeled threshold graph using the same method as above.
Also, any construction of a labeled threshold graph can be represented as a signed permutation.

\begin{example}
The construction of a labeled threshold graph associated to the signed permutation on $[5]$ given by $\overset{+}{2}\overset{-}{3}\overset{+}{4}\overset{+}{1}\overset{-}{5}$ is shown in Figure \ref{tgconst}.
\begin{figure}[H]
    \centering
    \begin{tikzpicture}[scale=0.5]
    \node (2) [circle,draw=black,inner sep=2pt] at (0,0) {\small $2$};
    \node (4) at (0.75,-2) {\small \color{white} 4};
    \node at (1.5,0.5) {$\xrightarrow{\overset{-}{3}}$};
    \end{tikzpicture}
    \begin{tikzpicture}[scale=0.5]
        \node (2) [circle,draw=black,inner sep=2pt] at (0,0) {\small $2$};
        \node (3) [circle,draw=black,inner sep=2pt] at (2,0) {\small $3$};
        \node (4) at (0.75,-2) {\small \color{white} 4};
        \node at (3.5,0.5) {$\xrightarrow{\overset{+}{4}}$};
    \end{tikzpicture}
    \begin{tikzpicture}[scale=0.5]
        \node (2) [circle,draw=black,inner sep=2pt] at (0,0) {\small $2$};
        \node (3) [circle,draw=black,inner sep=2pt] at (2,0) {\small $3$};
        \node (1) [circle,draw=black,inner sep=2pt] at (1,2) {\small $4$};
        \node (4) at (0.75,-2) {\small \color{white} 4};
        \node at (3.5,0.5) {$\xrightarrow{\overset{+}{1}}$};
        \draw (3)--(1)--(2);
    \end{tikzpicture}
    \begin{tikzpicture}[scale=0.5]
        \node (2) [circle,draw=black,inner sep=2pt] at (0,0) {\small $2$};
        \node (3) [circle,draw=black,inner sep=2pt] at (2,0) {\small $3$};
        \node (1) [circle,draw=black,inner sep=2pt] at (1,2) {\small $4$};
        \draw (3)--(1)--(2);
        \node at (3.5,0.5) {$\xrightarrow{\overset{-}{5}}$};
        \node (4) [circle,draw=black,inner sep=2pt] at (1,-2) {\small $1$};
        \draw (1)--(4);
        \draw (2)--(4);
        \draw (3)--(4);
    \end{tikzpicture}
    \begin{tikzpicture}[scale=0.5]
        \node (2) [circle,draw=black,inner sep=2pt] at (0,0) {\small $2$};
        \node (3) [circle,draw=black,inner sep=2pt] at (2,0) {\small $3$};
        \node (1) [circle,draw=black,inner sep=2pt] at (1,2) {\small $4$};
        \draw (3)--(1)--(2);
        \node (4) [circle,draw=black,inner sep=2pt] at (1,-2) {\small $1$};
        \node (5) [circle,draw=black,inner sep=2pt] at (3.5,0) {\small $5$};
        \draw (1)--(4);
        \draw (2)--(4);
        \draw (3)--(4);
    \end{tikzpicture}
    \caption{Construction of threshold graph corresponding to $\overset{+}{2}\overset{-}{3}\overset{+}{4}\overset{+}{1}\overset{-}{5}$.}
    \label{tgconst}
\end{figure}
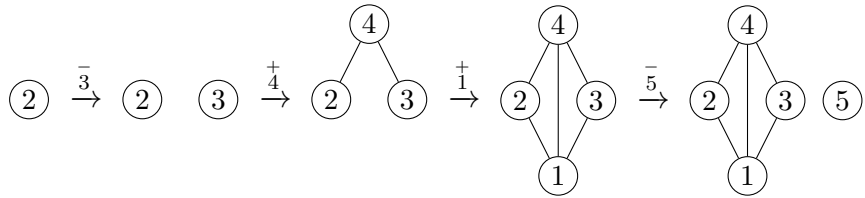
\end{example}

\begin{remark}
We make the following convention about constructions of labeled threshold graphs:
If $1$ is in the first block, it should be the first vertex added, and in this case $1$ is considered as a dominant vertex.
\end{remark}

This means that if $1$ is not the first vertex added, there should be some vertex added before it that is of different type (other than the first vertex) and when $1$ is added as the first vertex, it is considered to be a dominant vertex.

\begin{example}
The constructions $\overset{-}{2}\ \overset{+}{4}\ \overset{+}{1}\ \overset{-}{5}\ \overset{+}{3}$ and $\overset{-}{1}\ \overset{-}{5}\ \overset{-}{3}\ \overset{+}{2}\ \overset{+}{4}$ do not follow our convention.
Whereas, the constructions $\overset{+}{1}\ \overset{+}{4}\ \overset{+}{2}\ \overset{-}{5}\ \overset{+}{3}$ and $\overset{+}{1}\ \overset{-}{5}\ \overset{-}{3}\ \overset{+}{2}\ \overset{+}{4}$ do.
Any construction of the threshold graph constructed by $\overset{+}{4}\ \overset{+}{5}\ \overset{-}{2}\ \overset{+}{1}\ \overset{-}{3}$ will follow our convention since $1$ is not in the first block.
\end{example}

\begin{definition}
In a construction of a labeled threshold graph, an anchor is a vertex whose label is smaller than all vertices added after it.
\end{definition}

It is clear that the first anchor will always be $1$ and hence any construction always has an anchor.
The same threshold graph could have different anchors based on how it is constructed.

\begin{example}\label{ex4}
Consider the threshold graph corresponding to the threshold pair in standard form given by $\overset{-}{2}\ \overset{-}{3}\ \overset{-}{5}\ \overset{+}{1}\ \overset{+}{4}\ \overset{-}{6}$.
\begin{enumerate}
    \item The construction given by $\overset{+}{2}\ \overset{-}{3}\ \overset{-}{5}\ \overset{+}{1}\  \overset{+}{4}\ \overset{-}{6}$ has anchors $1,\ 4$ and $6$.
    \item The construction given by $\overset{-}{3}\ \overset{-}{2}\ \overset{-}{5}\ \overset{+}{4}\  \overset{+}{1}\ \overset{-}{6}$ has anchors $1$ and $6$.
\end{enumerate}
\end{example}

\begin{definition}
An odd anchor in a construction is  an anchor that is of one of the following forms:
\begin{enumerate}
    \item The first anchor (i.e., $1$) and dominant.
    \item An anchor that is of different type (i.e., dominant or recessive) from the previous anchor.
\end{enumerate}
\end{definition}

\begin{example}The odd anchors in the constructions given in \Cref{ex4} are $1$ and $6$ in both cases.
\begin{enumerate}
    \item In the construction $\overset{+}{6}\ \overset{-}{2}\ \overset{+}{3}\ \overset{-}{1}\ \overset{-}{5}\ \overset{+}{4}$, the anchors are $1$ and $4$ but only $4$ is an odd anchor.
    \item In the construction $\overset{+}{1}\ \overset{-}{3}\ \overset{-}{5}\ \overset{+}{2}\ \overset{+}{4}$ the anchors are $1,\ 2$ and $4$ but only $1$ is an odd anchor.
\end{enumerate}
\end{example}

\begin{lemma}\label{noa}
The number of odd anchors in any construction of a labeled threshold graph is the same, provided we follow the convention about the vertex $1$.
\end{lemma}
\begin{proof}
The existence of an anchor in a block does not depend on the construction.
This is because the least vertex in a block that contains an anchor will always be an anchor.
This also tells us that if a block contains an odd anchor in one construction, it does so in any construction.
In the case that $1$ is not added first, it can be shown that each block can have at most one odd anchor, the first anchor of the block, which proves our lemma in this case.

So, we have to deal with the case when $1$ is added first.
In this case, by our convention, $1$ is always an odd anchor.
Also, the second anchor, which is $2$, is an odd anchor if and only if it is recessive.
This property is invariant under choice of construction (following the convention) and the number of other odd anchors being the same is proved using the same logic as before.
\end{proof}

We can now prove our main theorem.

\begin{theorem}\label{thm}
The absolute value of the coefficient of $t^j$ in the polynomial $\chi_{\mathcal{T}_n}(t)$ is the number of threshold graphs with exactly $j$ odd anchors.
\end{theorem}
\begin{proof}
We will prove that the threshold graph corresponding to a threshold permutation with $j$ odd cycles has $j$ odd anchors.
The proof will then follow from \Cref{thm1}.
Anchors in the construction using a threshold pair in standard form correspond to the compartments of the corresponding threshold permutation (last elements of compartments).
So we will show that a compartment in a threshold permutation is an odd cycle if and only if its last element is an odd anchor in the construction given by the corresponding threshold pair in standard form.

We will be working with threshold permutations.
However, when we say `blocks', we mean the blocks of the threshold pair in standard from corresponding to that threshold permutation.
Similarly, when we say `anchor', `odd anchor', `dominant vertex', or `recessive vertex', we are talking about the construction given by the corresponding threshold pair in standard form.

Suppose that $1$ is not the first number in the threshold permutation, i.e., $1$ is not in the first block.
First we will show that the first compartment, whose last element will be $1$, is an odd cycle if and only if $1$ is a dominant vertex.
Suppose that the first block is a recessive block.
This means that if the first block is $B = \{b_1 ,b_2, \ldots, b_k\}$, where $b_1 < b_2 < \cdots < b_k$, the threshold permutation starts as
\begin{equation*}
    \overset{-}b_1\ \overset{-}b_2\ \overset{+}b_3\ \cdots\ \overset{+}b_k\ \overset{-}{c}\ \cdots
\end{equation*}
Since $1$ occurs after $b_k$, we see that the first compartment should have one of $3,5,7,\ldots$  $-$ signs in it for it to be odd.
But this corresponds to $1$ being in the $2^{nd},4^{th},6^{th}, \ldots$ block in the threshold block order (since every $-$ sign from the third one represents a change in block).
But the even numbered blocks are precisely the ones that contain dominant vertices.
A similar logic works if the first block is dominant.
So the first compartment is an odd cycle if and only if $1$ is dominant, i.e., $1$ is an odd anchor.

It can be shown that for all compartments other than the first, the condition that the anchor of the compartment is of different type from the previous anchor translates to there being an odd number of $-$ signs in the compartment.
Hence, in the case that $1$ is not in the first block, we see that the number of odd anchors is equal to the number of odd cycles.

Now consider the case when $1$ is in the first block.
In this case, $1$ is an odd anchor and also the compartment in the threshold permutation containing $1$ is odd.
We will show that the second compartment is odd if and only if the second anchor, which is $2$, is recessive.
Let the first block be $B = \{1 , b_1,\ldots, b_k\}$.
If the first block is dominant, the threshold permutation starts as:
\begin{equation*}
     \overset{-}1\ \overset{+}b_1\ \overset{+}b_2\ \cdots\ \overset{+}b_k\ \overset{-}{c}\ \cdots
\end{equation*}Since $\overset{-}1$ forms the first compartment, we see that the second compartment is odd if and only if there an odd number of $-$ signs after $1$ and before $2$.
But this precisely means that the $2$ is in one of the $2^{nd},4^{th},6^{th}, \ldots$ blocks and hence recessive.
A similar logic works if the first block is recessive.
The fact that the other compartments are odd if and only if the anchor in that compartment is odd follows just as before.
\end{proof}

The labeled threshold graphs with $2$ and $3$ vertices, constructed following the convention about the vertex $1$, and their corresponding number of odd anchors are given in \cref{tab:my_label2g} and \cref{tab:my_label3g} respectively.
The odd anchors of the graph are the ones colored red.

\FloatBarrier
\begin{table}[H]
    \addtolength{\tabcolsep}{10pt}
    \renewcommand{\arraystretch}{1.25}
    \centering
    \begin{tabular}{|M{2.5cm}|M{3cm}|M{2.5cm}|}
        \hline
        Construction & Threshold graph & Number of odd anchors\\
        \hline
        $\overset{+}{1}\ \overset{-}{2}$ &
        \begin{tikzpicture}
        \node at (0.5,0.3) {};
        \node (1) [circle,draw=black,inner sep=2pt,fill=red!10] at (0,0) {\small $1$};
        \node (2) [circle,draw=black,inner sep=2pt,fill=red!10] at (1,0) {\small $2$};
        \end{tikzpicture}
        & 2\\
        \hline
        $\overset{+}{1}\ \overset{+}{2}$ &
        \begin{tikzpicture}
        \node at (0.5,0.3) {};
        \node (1) [circle,draw=black,inner sep=2pt,fill=red!10] at (0,0) {\small $1$};
        \node (2) [circle,draw=black,inner sep=2pt] at (1,0) {\small $2$};
        \draw (1)--(2);
        \end{tikzpicture}
        & 1\\
        \hline
    \end{tabular}
    \caption{Labeled threshold graphs with $2$ vertices, each constructed following our convention.}
    \label{tab:my_label2g}
\end{table}
\FloatBarrier
\begin{table}[H]
    \addtolength{\tabcolsep}{10pt}
    \renewcommand{\arraystretch}{1.25}
    \centering
    \begin{tabular}{|M{2.5cm}|M{3cm}|M{2.5cm}|}
        \hline
        Construction & Threshold graph & Number of odd anchors\\
        \hline
        $\overset{+}{1}\ \overset{-}{2}\ \overset{+}{3}$ &
        \begin{tikzpicture}
        \node at (0.5,1.085) {};
        \node (1) [circle,draw=black,inner sep=2pt,fill=red!10] at (0,0) {\small $1$};
        \node (2) [circle,draw=black,inner sep=2pt,fill=red!10] at (1,0) {\small $2$};
        \node (3) [circle,draw=black,inner sep=2pt,fill=red!10] at (0.5,0.77) {\small $3$};
        \draw (2)--(3);
        \draw (1)--(3);
        \end{tikzpicture}
        & 3\\
        \hline
        $\overset{+}{1}\ \overset{+}{3}\ \overset{-}{2}$ &
        \begin{tikzpicture}
        \node at (0.5,1.085) {};
        \node (1) [circle,draw=black,inner sep=2pt,fill=red!10] at (0,0) {\small $1$};
        \node (2) [circle,draw=black,inner sep=2pt] at (1,0) {\small $3$};
        \node (3) [circle,draw=black,inner sep=2pt,fill=red!10] at (0.5,0.77) {\small $2$};
        \draw (1)--(2);
        \end{tikzpicture}
        & 2\\
        \hline
        $\overset{+}{1}\ \overset{+}{2}\ \overset{-}{3}$ &
        \begin{tikzpicture}
        \node at (0.5,1.085) {};
        \node (1) [circle,draw=black,inner sep=2pt,fill=red!10] at (0,0) {\small $1$};
        \node (2) [circle,draw=black,inner sep=2pt] at (1,0) {\small $2$};
        \node (3) [circle,draw=black,inner sep=2pt,fill=red!10] at (0.5,0.77) {\small $3$};
        \draw (1)--(2);
        \end{tikzpicture}
        & 2\\
        \hline
        $\overset{+}{1}\ \overset{-}{2}\ \overset{-}{3}$ &
        \begin{tikzpicture}
        \node at (0.5,1.085) {};
        \node (1) [circle,draw=black,inner sep=2pt,fill=red!10] at (0,0) {\small $1$};
        \node (2) [circle,draw=black,inner sep=2pt,fill=red!10] at (1,0) {\small $2$};
        \node (3) [circle,draw=black,inner sep=2pt] at (0.5,0.77) {\small $3$};
        \end{tikzpicture}
        & 2\\
        \hline
        $\overset{+}{1}\ \overset{+}{2}\ \overset{+}{3}$ &
        \begin{tikzpicture}
        \node at (0.5,1.085) {};
        \node (1) [circle,draw=black,inner sep=2pt,fill=red!10] at (0,0) {\small $1$};
        \node (2) [circle,draw=black,inner sep=2pt] at (1,0) {\small $2$};
        \node (3) [circle,draw=black,inner sep=2pt] at (0.5,0.77) {\small $3$};
        \draw (1)--(2);
        \draw (2)--(3);
        \draw (1)--(3);
        \end{tikzpicture}
        & 1\\
        \hline
        $\overset{+}{1}\ \overset{-}{3}\ \overset{+}{2}$ &
        \begin{tikzpicture}
        \node at (0.5,1.085) {};
        \node (1) [circle,draw=black,inner sep=2pt,fill=red!10] at (0,0) {\small $1$};
        \node (2) [circle,draw=black,inner sep=2pt] at (1,0) {\small $3$};
        \node (3) [circle,draw=black,inner sep=2pt] at (0.5,0.77) {\small $2$};
        \draw (2)--(3);
        \draw (1)--(3);
        \end{tikzpicture}
        & 1\\
        \hline
        $\overset{-}{2}\ \overset{-}{3}\ \overset{+}{1}$ &
        \begin{tikzpicture}
        \node at (0.5,1.085) {};
        \node (1) [circle,draw=black,inner sep=2pt] at (0,0) {\small $2$};
        \node (2) [circle,draw=black,inner sep=2pt] at (1,0) {\small $3$};
        \node (3) [circle,draw=black,inner sep=2pt,fill=red!10] at (0.5,0.77) {\small $1$};
        \draw (2)--(3);
        \draw (1)--(3);
        \end{tikzpicture}
        & 1\\
        \hline
        $\overset{+}{2}\ \overset{+}{3}\ \overset{-}{1}$ &
        \begin{tikzpicture}
        \node at (0.5,1.085) {};
        \node (1) [circle,draw=black,inner sep=2pt] at (0,0) {\small $2$};
        \node (2) [circle,draw=black,inner sep=2pt] at (1,0) {\small $3$};
        \node (3) [circle,draw=black,inner sep=2pt] at (0.5,0.77) {\small $1$};
        \draw (1)--(2);
        \end{tikzpicture}
        & 0\\
        \hline
    \end{tabular}
    \bigskip
    \caption{Labeled threshold graphs with $3$ vertices, each constructed following our convention.}
    \label{tab:my_label3g}
\end{table}

\section{Acknowledgements}
The authors sincerely thank Sam Spiro and Richard Stanley for comments on an earlier version of this paper and the anonymous referee for helpful suggestions.

\bibliographystyle{abbrv} 
\bibliography{refs} 

\end{document}